\documentclass[a4paper,11pt,english]{amsart}
\usepackage{amsmath,latexsym,amssymb,amsfonts}
\usepackage{amscd,amssymb,epsfig}
\usepackage[arrow, matrix, curve]{xy}
\usepackage[colorlinks=true, linkcolor=purple, urlcolor=blue, citecolor=teal]{hyperref}
\usepackage{mathdots}
\usepackage{multirow}
\usepackage{siunitx}
\usepackage{enumitem}
\usepackage{xcolor}
\usepackage[nameinlink]{cleveref}
\usepackage{IEEEtrantools}

\numberwithin{equation}{section}

\newtheorem{theorem}{Theorem}[section]
\newtheorem{lemma}[theorem]{Lemma}
\newtheorem{corollary}[theorem]{Corollary}
\newtheorem{proposition}[theorem]{Proposition}

\theoremstyle{definition}
\newtheorem{definition}[theorem]{Definition}

\newtheorem{remark}[theorem]{Remark}
\newtheorem{conjecture}{Conjecture}
\newtheorem{question}{Question}

\setlist[enumerate,1]{label={(\roman*)}}

\renewcommand{\epsilon}{\varepsilon}

\newcommand{\C}{\mathbb{C}}

\newcommand{\E}{\mathbb{E}}

\newcommand{\FF}{\mathbb{F}}

\newcommand{\M}{\mathcal{M}}
\newcommand{\N}{\mathbb{N}}

\newcommand{\R}{\mathbb{R}}
\newcommand{\Z}{\mathbb{Z}}

\newcommand{\fx}{\mathbf{x}}

\newcommand{\id}{\operatorname{id}}

\newcommand{\tr}{\operatorname{tr}}

\newcommand{\Tr}{\operatorname{Tr}}

\newcommand{\rank}{\operatorname{rank}}

\renewcommand{\1}{\mathbf{1}}

\newcommand{\NC}{\operatorname{NC}}

\newcommand{\capa}{\mathrm{cap}}
\newcommand{\sym}{{\mathrm{h}}}
\newcommand{\ds}{\mathrm{ds}}

\makeatletter
\def\moverlay{\mathpalette\mov@rlay}
\def\mov@rlay#1#2{\leavevmode\vtop{%
\baselineskip\z@skip \lineskiplimit-\maxdimen
\ialign{\hfil$#1##$\hfil\cr#2\crcr}}}
\makeatother

\def\plangle{\moverlay{(\cr<}}
\def\prangle{\moverlay{)\cr>}}

\usepackage{eqparbox}
\usepackage[numbers]{natbib}

\makeindex

\title[FK-determinants of matrix-valued semicircular elements]{Fuglede-Kadison determinants of matrix-valued semicircular elements and capacity estimates}

\author[T. Mai]{Tobias Mai}
\address{Saarland University, Department of Mathematics, D-66123 Saarbr\"ucken, Germany}
\email{mai@math.uni-sb.de}

\author[R. Speicher]{Roland Speicher}
\email{speicher@math.uni-sb.de}

\date{\today}

\begin{document}

\begin{abstract}
We calculate the Fuglede-Kadison determinant of arbitrary matrix-valued semicircular operators in terms of the capacity of the corresponding covariance mapping. We also improve a lower bound by Garg, Gurvits, Oliveira, and Widgerson on this capacity, by making it dimension-independent.
\end{abstract}

\maketitle

\section{Introduction}

Consider a matrix
$A=a_1\otimes \fx_1+\ldots+a_n\otimes \fx_n$, 
where $\fx_1,\dots,\fx_n$ are noncommuting variables and $a_1,\dots,a_n\in M_m(\C)$ are arbitrary matrices of the same size $m$.
The noncommutative Edmonds' problem asks whether $A$ is invertible over the noncommutative rational functions in $\fx_1,\dots,\fx_n$. This problem has become quite prominent in recent years and there are now a couple of deterministic algorithms to decide about this \cite{GGOW16, GGOW19, ivanyos2018constructive, hamada2020computing,chatterjee2023noncommutative}. In \cite{MSY2023,hoffmann2023computing} we provided a non-commutative probabilistic approach to this question by showing that this algebraic problem is equivalent to an analytic one, where we replace the formal variables $\fx_i$ by concrete operators on infinite-dimensional Hilbert spaces.

A particular nice choice for the analytic operators are freely independent semicircular variables $s_1,\dots,s_n$, which are the non-commutative analogues of independent Gaussian random variables. In the case where all the $a_i$ are selfadjoint, the matrix 
$S=a_1\otimes s_1+\ldots+a_n\otimes s_n$ is an operator-valued semicircular element, whose classical distribution $\mu_S$ can be described via an equation for its Cauchy transform. In order to check invertibility of $S$ (as an unbounded operator) one has to decide whether $\mu_S$ has an atom at zero.  Since the equation for the Cauchy transform cannot be solved explicitly we have to rely on numerical approximations and then the question becomes relevant how we can distinguish an atom at zero from a density with high accumulation of mass about zero. What we need is a quantity which specifies this information and which can be calculated or at least controlled without having access to full knowledge of the distribution.

It should be evident that for ordinary, finite-dimensional, matrices such a quantity is given by the determinant. We will show in this note that this insight can be extended to our infinite-dimensional setting.

Let us first have a look why the determinant is for a matrix $B\in M_d(\C)$ a good tool to decide upon how many eigenvalues we might have close to zero. Let us restrict to positive matrices, so that the eigenvalues are also positive. It is then clear that if we have a lower bound on the determinant and an upper bound on the eigenvalues then we can also 
obtain an upper estimate for the number of eigenvalues in any given interval about zero. Upper bounds on the eigenvalues are usually easy to achieve via bounding the operator norm, so the main question is on lower bounds for the determinant. 

By scaling we can, of course, move as many eigenvalues as we like close to zero, so we need some normalization if we hope for a uniform control of the accumulation close to zero. Hence, we only consider matrices where all the entries are integers. The question is whether we also have a uniform lower bound on their determinant. Indeed we do. The determinant of such an integer matrix must necessarily be an integer. If we consider only invertible positive matrices then this integer must be at least 1. 

Putting these observations together indicates that no invertible matrix with integer entries can have too many eigenvalues close to zero, and the precise estimate depends on the matrix only through its operator norm.

Our main message is now that these considerations survive in a nice form also in the infinite dimensional setting. Whereas the determinant itself cannot be generalized to infinite dimensions, the following variant
$$\Delta(B):=\det(BB^*)^{\frac 1{2d}}\qquad \text{for $B\in M_d(\C)$}$$
has such an extension.
This quantity was introduced by Fuglede and Kadison in \cite{fuglede1952determinant}, and is usually called the Fuglede-Kadison determinant. It has all the nice properties one expects of a determinant and has become quite prominent in a variety of contexts, see, for example, \cite{deninger2009mahler, haagerup2000brown, hayes2016fuglede, luck2002l2} 

We can apply this determinant to our operator-valued semicircular operators $S=a_1\otimes s_1+\ldots+a_n\otimes s_n$ and see what we get from there. We will show that the invertibility of $S$ means, as in the finite-dimensional case, that $\Delta(S)\not=0$. Having a lower bound on $\Delta(S)$ gives an upper estimate on how much mass the distribution of $S$ can have in any given interval about zero. So the main question is whether the reasoning from above for lower bounds survives somehow in this setting. For a normalization we will now require that all matrices $a_i\in M_m(\C)$ have only integer entries. Since in the non-commutative setting there is now no explicit formula for the calculation of $\Delta(S)$ in terms of entries of $S$, there is no apriori reason that $\Delta(S)$ should be an integer, or of any other special form. However, we will show that indeed there is such a reason.

Let us first remark that there are other situations where the finite-dimensional arguments have been shown to work. Namely, if we replace the $\fx_i$ not by free semicirculars, but by free Haar unitaries instead, then
everything still works. Free Haar unitaries are the generators of the free group in their left regular representation and they can be approximated in distribution by permutation matrices. Since the entries of those permutation matrices are integers, namely $0$ or $1$, we have the estimates as before for the permutation matrices and, by upper semi-continuity of the Fuglede-Kadison determinant, this goes also over to the limiting Haar unitaries. This kind of argument works actually not only for free groups, but for all sofic groups, which can be approximated by permutation matrices in some sense. In this group context such questions have become quite prominent in the last decades. In particular, the use of the Fuglede-Kadison determinant in this context was promoted by L\"uck \cite{luck2002l2}; he also formulated the lower bound on the determinant as the ``determinantal conjecture''. There was much work on this, see, for example, \cite{schick2001}, before it was proved finally for all sofic groups by Elek and Szabo in \cite{elek2005hyperlinearity}. Those investigations are related with questions about Atiyah properties, Betti numbers, or Mahler measures.

Though this looks very close to what we want, there is no way of putting our setting of an operator-valued semicircular operator in the group context and we also see no direct way of changing the free Haar unitaries into free semicirculars in a way that allows control of the determinant. However, as we will show, we can make this transition from Haar unitaries to semicirculars in a more indirect way. This relies essentially on the notion of ``capacity'' of completely positive maps, which was introduced by Gurvits \cite{Gur2004} and crucially used by Garg, Gurvits, Oliveira, and Widgerson in the analysis of their algorithm in \cite{GGOW19}. In a sense we transfer their ideas to our setting, but along the way we also improve one of their main estimates on the capacity.

Let us make a remark on nomenclature. In the following our main concern are operators of the form $S=\sum_{i=1}^n a_i\otimes s_i$, where $s_1,\dots,s_n$ are free semicircular variables. If all the $a_i$ are selfadjoint then $S$ is also selfadjoint, and is a matrix-valued semicircular operator, in the sense of operator-valued free probability. If the $a_i$ are not necessarily selfadjoint, then $S$ is a non-selfajoint version of a matrix-valued semicircular element. For simplicity of language, we will address the general situation in the following just as matrix-valued semicircular element; should we insist on $S$ being selfajoint, then we will say so.

Let us now present our main statements. We begin with the notion of capacity.

\begin{definition}
For a (not necessarily completely) positive map
$\eta:M_m(\C)\to M_m(\C)$
we define its \emph{capacity} by
$$\capa(\eta):=\inf\{\det(\eta(b))\mid b\in M_m(\C), b>0, \det(b)=1\}.$$
\end{definition}

\begin{theorem}\label{thm:1.2}
Consider a matrix-valued semicircular element of the form
$S=\sum_{i=1}^n a_i\otimes s_i$ with $a_i\in M_m(\C)$,
and let $\eta$ be the corresponding covariance map, that is, the completely positive map
$$\eta:M_m(\C)\to M_m(\C), \quad b\mapsto \eta(b):=\sum_{i=1}^n a_iba^*_i.$$
Then the Fuglede-Kadison determinant of $S$
is calculated as follows:
$$\Delta(S)=\capa(\eta)^{\frac 1{2m}}e^{-\frac{1}{2}}.$$
\end{theorem}

If $S$ is not selfadjoint, we can also consider the dual covariance map
$$\eta^*:\ M_m(\C)\to M_m(\C), \quad b\mapsto \eta^*(b):=\sum_{i=1}^n a^*_iba_i.$$
Since $\Delta(S^*)=\Delta(S)$ and 
$\capa(\eta^*)=\capa(\eta)$, this does not contain any new information about the Fuglede-Kadison determinant of $S$.

An immediate consequence of \Cref{thm:1.3} is that $\capa(\eta)>0$ implies $\int_0^\infty \log t\, d\mu_{\vert S\vert}(t) > -\infty$. We point out that this follows also from \cite{SS2015}; however, \Cref{thm:1.3} goes far beyond this qualitative statement as it provides an explicit formula relating $\Delta(S)$ and $\capa(\eta)$ and so a quantification of the result of \cite{SS2015} in the case of matrix-valued semicircular elements. The approach of \cite{SS2015} involves the so-called \emph{Novikov-Shubin invariant} of $S$, for which one does not have sufficient control; it was in fact this problem which also occurred in \cite{hoffmann2023computing} and which has motivated the present paper.

If we want to use the formula from \Cref{thm:1.2} to obtain a lower bound on $\Delta(S)$ then we need such lower bounds on capacities. Such an estimate, with lower bound $1/(m^{2m})$, was given in Theorem 2.18 of \cite{GGOW19}. We will actually show that this can be improved to a bound that is independent of the size of the considered matrices.

\begin{theorem}\label{thm:1.3}
Let $\eta$ be a completely positive map with integer coefficients, that is,
$$\eta:\ M_m(\C)\to M_m(\C), \quad b\mapsto \eta(b):=\sum_{i=1}^n a_iba^*_i,\qquad
\text{with $a_i\in M_m(\Z)$}.$$ 
If $\eta$ is rank non-decreasing then
$\capa(\eta)\geq 1$.
\end{theorem}

The proof of \Cref{thm:1.3} will be given in \Cref{sec:proof_of_thm:1.3}.
The condition that $\eta$ is rank non-decreasing is equivalent to the invertibility of $A = a_1 \otimes \fx_1 + \dots + a_n \otimes \fx_n$ over the noncommutative rational functions in $\fx_1,\dots,\fx_n$, namely the so-called \emph{free field} $\C\plangle \fx_1,\dots,\fx_n\prangle$; see \Cref{subsec:capacity} and \cite{GGOW19}. From \cite{MSY2023}, we learn that $A$ is invertible over $\C\plangle \fx_1,\dots,\fx_n\prangle$ if and only if the associated matrix-valued semicircular operator $S = a_1 \otimes s_1 + \dots + a_n \otimes s_n$ is invertible as an unbounded operator. Thus, \Cref{thm:1.3} gives the following.

\begin{corollary}\label{thm:1.4}
Consider, for integer matrices $a_i\in M_m(\Z)$, the matrix-valued semicircular operator
$S=\sum_{i=1}^n a_i\otimes s_i$.
If $S$ is invertible as an unbounded operator, then we have for its Fuglede-Kadison determinant that
$\Delta(S)\geq e^{-\frac{1}{2}}$.
\end{corollary}

This is used in \cite{hoffmann2023computing} to achieve a uniform control of how much mass can accumulate at zero in the distribution of such matrix-valued semicircular operators.

\section{Preliminaries}

\subsection{Fuglede-Kadison determinant}

\begin{definition}
Let $(\M,\tau)$ be a tracial $W^*$-probability space; this means that $\M$ is a finite von Neumann algebra and $\tau:\M\to\C$ is a faithful, normal, tracial state.

\begin{enumerate}
\item 
Let $X\in\M$ be a selfadjoint operator. Then the \emph{distribution} $\mu_X$ of $X$ is the probability measure on $\R$ that is uniquely determined by
$$\tau(X^k)=\int_\R t^kd\mu_X(t)\qquad \text{for all $k\in\N$}.$$

\item 
For a (not necessarily selfadjoint) operator $T\in\M$ its \emph{Fuglede-Kadison determinant} is defined by
$$\Delta(T):=\exp\left(\int_0^\infty \log t\, d\mu_{\vert T\vert}(t)\right)\in [0,\infty),$$
where $\vert T\vert:=(T^*T)^{1/2}$. 
\end{enumerate}
\end{definition}

\begin{remark}\label{rem:FK_definition}
\begin{enumerate}[wide]
    \item\label{it:FK_definition_kernel}
    In the integral $\log (0)$ has to be interpreted as $-\infty$. This means that in case that $T$ has a non-trivial kernel its Fuglede-Kadison determinant is equal to zero. 

    \item
    One should note, however, that the Fuglede-Kadison determinant of $T$ can be zero even if $T$ has trivial kernel. This means that the distribution of $\vert T\vert$ has high accumulation of mass near zero without having an atom there.

    \item 
    We will in particular be interested in operators of the form $T=\sum_{i=1}^n a_i\otimes x_i$, where the $a_i\in M_m(\C)$ are just ordinary matrices, equipped with the normalized trace $\tr_m=\frac 1m\Tr_m$, and $x_1,\dots,x_n$ are operators from a tracial $W^*$-probability space $(\M,\tau)$. The operator $T$ is then considered, and its determinant $\Delta(T)$ calculated, in the tracial $W^*$-probability space $(M_m(\M),\tr_m\otimes \tau)$, where $M_m(\M)=M_m(\C)\otimes\M$. Note that even though we take now the determinant of a matrix with operators from $\M$ as entries, this determinant is not an element in $\M$, but just a non-negative real number. In particular, there is no concrete formula to calculate the determinant in terms of the entries of the matrix.

    \item 
     We will usually consider matrix-valued semicircular operators $S=\sum_{i=1}^n a_i\otimes s_i$ which have zero kernel, and thus have a (possibly unbounded) inverse that is affiliated with its tracial von Neumann algebra. 
    As stated above the triviality of the  kernel does not mean that $\Delta(S)\not=0$. It will however be one of our results that for the special operators we consider this is indeed the case.

    \item 
    One can also define the Fuglede-Kadison determinant $\Delta(S^{-1})$ for the typically unbounded inverse $S^{-1}$. Haagerup and Schulz developed in \cite{haagerup2007brown} a general theory of $\Delta$ for unbounded affiliated operators. We will however not need this in our calculations and estimates, since we never use the inverse of $S$, but only apply $\Delta$ to inverses of finite-dimensional matrices from $M_m(\C)$. 
    \end{enumerate}
\end{remark}

In the following propositions we collect the relevant properties of the Fuglede-Kadison determinant.

\begin{proposition}\label{prop:FK_determinant_properties}
 Let $(\M,\tau)$ be a tracial $W^*$-probability space and $\Delta:\M\to [0,\infty)$ the corresponding Fuglede-Kadison determinant.   

\begin{enumerate}
  
   \item\label{it:FK_determinant_multiplicative}
 The Fuglede-Kadison determinant is multiplicative; that is, we have for all $S,T\in \M$ that
    $$\Delta(ST)=\Delta(S)\Delta(T).$$

     \item\label{it:FK_determinant_adjoint}
     For $T\in\M$ we have
     $$\Delta(T)=\Delta(T^*).$$

\item\label{it:FK_AMGM}
For a positive operator $T\in\M$ we have the following version of the geometric mean-arithmetic mean inequality: 
$$\Delta(T)\leq \tau(T).$$

    \item\label{it:FK_determinant_tensor}
    For $T\in\M$ and $b\in M_m(\C)$ we have
     $$\Delta(\1_m\otimes T)=\Delta(T),\qquad
     \Delta(b\otimes \1)=\Delta(b).$$
     Here, the Fuglede-Kadison determinants of $\1_m\otimes T, b\otimes \1 \in M_m(\M)$ are of course calculated with respect to $\tr_m\otimes\tau$, whereas $\Delta(b)$ is calculated with respect to $(M_m(\C),\tr_m)$.

    \item\label{it:FK_determinant_matrix}
    For $T_1,T_2\in\M$ we have 
    $$\Delta\left(\begin{bmatrix}
        T_1&0\\
        0&T_2
    \end{bmatrix}\right)= 
    \Delta\left(\begin{bmatrix}
        0& T_1\\
        T_2&0
    \end{bmatrix}\right)=
    \Delta(T_1)^{1/2} \cdot\Delta(T_2)^{1/2}.$$
    Here the first two Fuglede-Kadison determinants are of course calculated in $M_2(\M)$ with respect to $\tr_2\otimes \tau$.
    
\item\label{it:FK_determinant_matrices}
If $(\M,\tau)=(M_m(\C),\tr_m)$ then we have for a finite-dimensional matrix $b\in M_m(\C)$
$$\Delta(b)=\det(bb^*)^{\frac 1{2m}}=\det{(\vert b\vert)^{\frac 1m}}.$$

\item\label{it:FK_determinant_doubly_stochastic}
Let $s\in\M$ be a standard semicircular element, meaning that $s$ is a selfadjoint operator the distribution $\mu_s$ of which lives on the interval $[-2,2]$ and is given there by the density
$$d\mu_s(t)=\frac 1{2\pi} \sqrt{4-t^2}\, dt.$$
Then 
$$\Delta(s)=e^{-\frac{1}{2}}.$$

\end{enumerate}
    \end{proposition}

    \begin{proof}
    Most of those properties are already contained in the original paper \cite{fuglede1952determinant}; however, since we need them also for ''singular'' operators, which have no bounded inverse, some of them are hidden in Section 5 of \cite{fuglede1952determinant}, and we will provide in the following also more explicit references.
        \begin{enumerate}
            \item 
            See Prop. 2.5 in \cite{haagerup2007brown} or Remark 2.3 in \cite{nayak2018hadamard}.

            \item 
This is clear, since $\vert T^*\vert$ and $\vert T\vert$ have in our tracial setting the same distribution.

\item 
See Lemma 3.4. in \cite{nayak2018hadamard}.

\item 
This follows directly from the fact that the distribution of $\vert \1_m\otimes T\vert=\1_m\otimes\vert T\vert$ is the same as the distribution of $\vert T\vert$, and the distribution of $\vert b\otimes \1\vert=\vert b\vert\otimes \1$ is the same as the distribution of $\vert b\vert$.

            \item 
            For the calculation of the first term, see Remark 3.5 in \cite{nayak2018hadamard}.
The statement for the second term follows from this by observing that the absolute values of both matrices are the same, or, more explicitly, as follows.

        \begin{align*}
        \Delta\left(\begin{bmatrix}
        0& T_1\\
        T_2&0
    \end{bmatrix}\right)&= 
    \Delta\left(
\begin{bmatrix}
        0&1\\
        1&0
    \end{bmatrix}
\begin{bmatrix}
        T_2&0\\
        0&T_1
    \end{bmatrix}
    \right)\\[0.3ex]
    &=
\Delta\left(
\begin{bmatrix}
        0&1\\
        1&0
    \end{bmatrix}\right)\cdot
    \Delta\left(
\begin{bmatrix}
        T_2&0\\
        0&T_1
    \end{bmatrix}
    \right)\\[0.3ex]
    &=
    \Delta(T_2)^{1/2} \cdot\Delta(T_1)^{1/2}.
        \end{align*}

            \item 
            See Example 2.4 in \cite{nayak2018hadamard}.
            
            \item 
            This follows from the fact that $\vert s\vert$ has a quarter-circular distribution on $[0,2]$ with density
            $$d\mu_{\vert s\vert}(t)=\frac 1\pi \sqrt{4-t^2}\, dt$$
            and that 
            \[
            \frac 1\pi \int_0^2 \log t\, \sqrt{4-t^2}\, dt=-\frac 12. \qedhere
            \]
            
        \end{enumerate}
            
    \end{proof}

Crucial for most arguments is the following upper semi-continuity property of $\Delta$. 
The following version with respect to convergence in distribution goes back to \cite{schick2001}; a more explicit formulation and proof can be found in \cite{sofic}.
One should note that in this group context L\"uck \cite{luck2002l2} introduced a modified version of the Fuglede-Kadison determinant, which ignores the kernel of the operator. Thus we require in the following explicitly that all involved operators have trivial kernel, which guarantees that our determinant is the same as the modified one.

\begin{proposition}\label{prop:FK_determinant_lower_semicontinuity}
    Let $T_n$ ($n\in\N$) and $T$ be positive operators in some tracial $W^*$-probability space $(\M,\tau)$, such that $(T_n)_{n\in\N}$ converges in distribution to $T$, that is
$$\lim_{n\to\infty}\tau(T_n^k)=\tau(T^k)\qquad\text{for all $k\in\N$}.$$
Assume that all $T_n$ and $T$ have trivial kernel and that $\sup_n\Vert T_n\Vert\leq\infty$. Then we have
$$\Delta(T)\geq \limsup_{n\to\infty} \Delta(T_n).$$
\end{proposition}

\subsection{Capacity}\label{subsec:capacity}

Let $\eta: M_m(\C) \to M_m(\C)$ be a linear map. We denote by $\eta^\ast: M_m(\C) \to M_m(\C)$ the \emph{dual} to $\eta$ with respect to the inner product $\langle\cdot,\cdot\rangle$ on $M_m(\C)$ which is defined by $\langle x,y \rangle := \tr_m(xy^\ast)$, that is,
$$\langle \eta(x),y \rangle = \langle x, \eta^\ast(y) \rangle \qquad\text{for all $x,y\in M_m(\C)$}.$$

For matrices $b_1,b_2 \in M_m(\C)$, we write $b_2 \geq b_1$ if $b_2-b_1$ is positive semidefinite and we write $b_2>b_1$ if $b_2-b_1$ is positive definite.

\begin{definition}\label{def:positive_linear_maps}
A linear map $\eta: M_m(\C) \to M_m(\C)$ is said to be
\begin{enumerate}
 \item \emph{positive} if $\eta(b) \geq 0$ for all $b\geq 0$;
 \item \emph{completely positive} if all ampliations $\eta^{(k)}: M_k(\C) \otimes M_m(\C) \to M_k(\C) \otimes M_m(\C)$ with $\eta^{(k)} := \id_{M_k(\C)} \otimes \eta$ for $k\in \N$ are positive;
 \item \emph{strictly positive} if there exists $\alpha>0$ such that $\eta(b) \geq \alpha \tr_m(b) \1_m$ holds true for all $b\geq 0$.
\end{enumerate}
A positive linear map $\eta: M_m(\C) \to M_m(\C)$ is said to be
\begin{enumerate}
\setcounter{enumi}{3}
 \item \emph{rank non-decreasing} if $\rank(\eta(b)) \geq \rank(b)$ for all $b\geq 0$;
 \item\label{it:indecomposable} \emph{indecomposable} if we have $\rank(\eta(b)) > \rank(b)$ for all $b\geq 0$ which satisfy $1\leq \rank(b) < m$;
 \item \emph{doubly stochastic} if both $\eta(\1_m)=\1_m$ and $\eta^\ast(\1_m)=\1_m$.
\end{enumerate}
\end{definition}

\begin{remark}\label{rem:positive_linear_maps}
\begin{enumerate}[wide]
 \item\label{it:semi-flat_vs_strictly_positive} Linear maps $\eta: M_m(\C) \to M_m(\C)$ which are strictly positive have also appeared in the literature under the name \emph{semi-flat}; see \cite{AjEK2019,MSY2018,BM2020,BM2023}, for instance.
 \item\label{it:strict_positivity_implies_indecomposability} Note that strict positivity implies indecomposability. Indeed, if there is $\alpha>0$ such that $\eta(b) \geq \alpha \tr_m(b) \1_m$ holds true for all $b\geq 0$, then $\eta(b)$ has full rank for all positive semidefinite matrices $0 \neq b \in M_m(\C)$; hence, we have $\rank(\eta(b)) = m > \rank(b)$ for each $b\geq 0$ with $1\leq \rank(b) < m$.
 \item Clearly, indecomposable positive linear maps are rank non-decreasing.
\end{enumerate}
\end{remark}

According to the definitions, $\capa(\eta)>0$ ensures that there is no positive matrix which is mapped to a singular matrix by $\eta$, whereas a rank non-decreasing $\eta$ does not reduce the rank of any positive semidefinite matrix. Remarkably, these two conditions are still equivalent.

\begin{proposition}[Lemma 4.5 in \cite{Gur2004}]\label{prop:capacity_positive}
A positive linear map $\eta: M_m(\C) \to M_m(\C)$ is rank non-decreasing if and only if $\capa(\eta) > 0$.
\end{proposition}

Note that the infimum in the definition of the capacity is not a minimum in general. In the indecomposable case, however, the infimum is attained and the minimizer is unique. It follows that any indecomposable map can be transformed via operator scaling into a doubly stochastic positive linear map. The precise statement is given in the next proposition.

Following \cite[Definition 4.2]{Gur2004} and \cite{GGOW19}, we define for a positive linear map $\eta: M_m(\C) \to M_m(\C)$ and arbitrary matrices $c_1,c_2 \in M_m(\C)$ the \emph{operator scaling} $\eta_{c_1,c_2}: M_m(\C) \to M_m(\C)$ of $\eta$ as the positive linear map determined by
$$\eta_{c_1,c_2}(b) := c_1 \eta(c_2^\ast b c_2) c_1^\ast.$$
Note that $(\eta_{c_1,c_2})^\ast = (\eta^\ast)_{c_2,c_1}$.

\begin{proposition}[Theorem 4.7 in \cite{Gur2004}]\label{prop:indecomposable_minimizer}
A positive linear map $\eta: M_m(\C) \to M_m(\C)$ is indecomposable if and only if there exists a unique positive definite matrix $c \in M_m(\C)$ with $\det(c)=1$ and $\det(\eta(c)) = \capa(\eta)$. In this case, the operator scaling $\eta_{\eta(c)^{-1/2},c^{1/2}}$ of $\eta$ is doubly stochastic.
\end{proposition}

Note that if $\eta: M_m(\C) \to M_m(\C)$ is rank non-decreasing, then automatically $(\eta_{\eta(c)^{-1/2},c^{1/2}})^\ast(\1_m) = \1_m$ for all positive definite matrices $c \in M_m(\C)$.

For two positive linear maps $\eta_0,\eta_1: M_m(\C) \to M_m(\C)$, we write $\eta_1 \geq \eta_0$ if the linear map $\eta_1 - \eta_0: M_m(\C) \to M_m(\C)$ is positive. We note that capacity is monotonic with respect to this partial order on positive linear maps and we even have that the function $\eta \mapsto \capa(\eta)^{1/m}$ is superadditive on the set of all positive linear maps $\eta: M_m(\C) \to M_m(\C)$. This result is certainly well-known, but since we could not locate it in the literature, we include the proof.

\begin{proposition}\label{prop:capacity_monotonicity}
Let $\eta_0,\eta_1: M_m(\C) \to M_m(\C)$ be positive linear maps satisfying $\eta_1 \geq \eta_0$. Then
$$\capa(\eta_1)^{1/m} \geq \capa(\eta_1-\eta_0)^{1/m} + \capa(\eta_0)^{1/m}$$
and in particular $\capa(\eta_1) \geq \capa(\eta_0)$.
\end{proposition}

\begin{proof}
Let $b \in M_m(\C)$ with $b > 0$ and $\det(b) =1$ be given. With the help of Minkowski's determinant inequality (see Theorem 13.5.4 and Exercise 13.5.1 in \cite{Mirsky1955}, for instance), we get that
\begin{multline*}
\det(\eta_1(b))^{1/m} \geq \det((\eta_1-\eta_0)(b))^{1/m} + \det(\eta_0(b))^{1/m}\\ \geq \capa(\eta_1-\eta_0)^{1/m} + \capa(\eta_0)^{1/m}
\end{multline*}
and hence, by passing to the infimum over all $b \in M_m(\C)$ with $b > 0$ and $\det(b) =1$, that $\capa(\eta_1)^{1/m} \geq \capa(\eta_1-\eta_0)^{1/m} + \capa(\eta_0)^{1/m}$, as asserted. The inequality $\capa(\eta_1) \geq \capa(\eta_0)$ is an immediate consequence of the latter because $\capa(\eta_1-\eta_0) \geq 0$ holds true by definition of capacity; alternatively, this can be deduced from the fact (see Exercise 12 in Section 82 of \cite{Halmos1974}, for instance) that $\det(\eta_1(b)) \geq \det(\eta_0(b))$ for all $b \in M_m(\C)$ with $b \geq 0$.
\end{proof}

In order to measure how much a positive linear map $\eta$ deviates from being doubly stochastic, we use the quantity
$$\ds(\eta) := \Tr_m\big((\eta(\1_m)-\1_m)^2\big) + \Tr_m\big((\eta^\ast(\1_m)-\1_m)^2\big)$$
introduced in \cite[Definition 2.3]{Gur2004}. The following proposition tells us that approximate minimizers of the infimum defining the capacity $\capa(\eta)$ can be used to transform $\eta$ by operator scaling into completely positive linear maps $\widetilde{\eta}$ for which $\ds(\widetilde{\eta})$ is arbitrarily small; this can be seen as a quantified version of \Cref{prop:indecomposable_minimizer}.

\begin{proposition}[Lemma 3.10 in \cite{GGOW19}]\label{prop:approximate_minimizer}
Let $\eta: M_m(\C) \to M_m(\C)$ be completely positive with $\capa(\eta) > 0$. Suppose that $c \in M_m(\C)$ is a positive definite matrix such that
$$\capa(\eta) \geq e^{-\delta} \cdot \frac{\det(\eta(c))}{\det(c)},$$
for some $\delta \in (0,\frac{1}{6}]$. Then
$$\ds(\eta_{\eta(c)^{-1/2},c^{1/2}}) = \Tr_m\big((c \eta^\ast(\eta(c)^{-1}) - \1_m)^2\big) \leq 6\delta.$$
\end{proposition}

\section{Proof of the formula for the FK-determinant of a matrix-valued semicircular element}

In this section, we present the proof of \Cref{thm:1.2}. The rough idea is to change arbitrary matrix-valued semicircular elements to doubly stochastic ones, namely matrix-valued semicircular elements for which the associated covariance maps are doubly stochastic. This is done either exactly (in the indecomposable case treated in \Cref{subsec:proof_thm:1.2_indecomposable}) or approximately (in the general case, discussed in \Cref{subsec:proof_thm:1.2_lower_bound} and \Cref{subsec:proof_thm:1.2_upper_bound}).
The special role of doubly stochastic matrix-valued semicircular elements is explained in \Cref{subsec:proof_thm:1.2_doubly_stochastic}.

\begin{remark}\label{rem:vanishing_capacity}
Some parts of this strategy to prove \Cref{thm:1.2} only work for those matrix-valued semicircular elements $S=a_1 \otimes s_1 + \dots + a_n \otimes s_n$, for which the associated covariance map $\eta$ satisfies $\capa(\eta) > 0$. This, however, is not a restriction: If $\capa(\eta)=0$, then the matrix $A = a_1 \otimes \fx_1 + \dots + a_n \otimes \fx_n$ in $M_m(\C\langle \fx_1,\dots,\fx_n\rangle)$ fails to be invertible over the free field $\C\plangle \fx_1,\dots,\fx_n\prangle$. Thus, \cite[Theorem 1.1]{MSY2023} tells us that $S$ and consequently $|S|$ have non-trivial kernels. This means that the distribution of $|S|$ has an atom at $0$, so that $\Delta(S)=0$ by \Cref{rem:FK_definition} \ref{it:FK_definition_kernel}. Thus, the formula asserted in \Cref{thm:1.2} is true in the case $\capa(\eta)=0$. 
\end{remark}

\subsection{Doubly stochastic matrix-valued semicircular elements}\label{subsec:proof_thm:1.2_doubly_stochastic}

Selfadjoint doubly stochastic matrix-valued semicircular elements are special, because one knows (see Theorem 3.3 in \cite{nica2002operator}, or also the discussion in Section 9.4 of \cite{MS17}) that their scalar-valued distribution is just an ordinary semicircle, whose Fuglede-Kadison determinant is given by \Cref{prop:FK_determinant_properties} \ref{it:FK_determinant_doubly_stochastic}.

\begin{proposition}\label{prop:doubly_stochastic_semicircular}
Let $S = a_1 \otimes s_1 + \ldots + a_n \otimes s_n$ be a matrix-valued semicircular
element, with selfadjoint $a_i=a_i^*\in M_m(\C)$ and corresponding covariance map
$$\eta:\ M_m(\C) \to M_m(\C), \quad b\mapsto 
\eta(b) = \sum^n_{i=1} a_i b a_i.$$
Note that in this case $S$ is selfadjoint and $\eta^*=\eta$.

If $\eta$ is doubly stochastic, that is, if $\eta(\1_m)=\1_m$, then the distribution of $S$ is the standard semicircle distribution and thus its Fuglede-Kadison determinant is given by $\Delta(S)=e^{-\frac{1}{2}}$.
\end{proposition}

Note that for doubly stochastic matrix-valued semicircular elements which are not selfadjoint, \Cref{lem:opval_semicricular_operations} \ref{it:opval_semicricular_operations:hermitization} will provide us a general method for transforming them in such way that \Cref{prop:doubly_stochastic_semicircular} can be applied.

We will also need the following approximate version of  \Cref{prop:doubly_stochastic_semicircular}.

\begin{proposition}\label{prop:approximate_doubly_stochastic_semicircular}
Let $S = a_1 \otimes s_1 + \ldots + a_n \otimes s_n$ be a selfadjoint matrix-valued semicircular element with the associated self-dual covariance map $\eta: M_m(\C) \to M_m(\C), b\mapsto \sum^n_{i=1} a_i b a_i$. Then, for all $k \in \N$, we have
$$\left|(\tr_m \otimes \tau)(S^{2k}) - \int^2_{-2} t^{2k}\, d\mu_s(t) \right| \leq C_k \bigg(\sum^{k-1}_{j=0} \|\eta\|^j\bigg) \|\eta(\1_m) - \1_m\|,$$
where $C_k$ denotes the $k$-th Catalan number and $\mu_s$ is the distribution of a standard semicircular element $s$ (see \Cref{prop:FK_determinant_properties} \ref{it:FK_determinant_doubly_stochastic}).
\end{proposition}

Let us point out that the assertion of \Cref{prop:approximate_doubly_stochastic_semicircular} extends trivially to the case $k=0$. Moreover, we note that the odd moments of $S$ satisfy $(\tr_m \otimes \tau)(S^{2k+1}) = 0 = \int^2_{-2} t^{2k+1}\, d\mu_s(t)$ for all $k \in \N_0$. Thus, \Cref{prop:approximate_doubly_stochastic_semicircular} shows that $S$ converges in distribution to $s$ as $\eta(\1_m)$ converges to $\1_m$.

\begin{proof}[Proof of \Cref{prop:approximate_doubly_stochastic_semicircular}]
Recall that $S$ is an operator-valued semicircular element in the operator-valued probability space $(M_m(\M),\E,M_m(\C))$ with the conditional expectation $\E := \id_{M_m(\C)} \otimes \tau$. Therefore, the moment-cumulant formula (see, for instance, Definition 9.7 in \cite{MS17}) for $S$ tells us that
\begin{equation}\label{eq:moment_cululant_formula}
b_0 \E[S b_1 S  \cdots b_{2k-1} S] b_{2k} = \sum_{\pi \in \NC_2(2k)} \eta_\pi(b_0 \otimes b_1 \otimes \dots \otimes b_{2k-1} \otimes b_{2k}),
\end{equation}
where $\NC_2(2k)$ denotes the set of all non-crossing pairings of $\{1,\dots,2k\}$ and the linear map $\eta_\pi: M_m(\C)^{\otimes (2k+1)} \to M_m(\C)$ is defined by applying $\eta$ to its arguments according to the block structure of $\pi \in \NC_2(2k)$. Note that
\begin{equation}\label{eq:eta_pi-bound}
\|\eta_\pi(b_0 \otimes \dots \otimes b_{2k})\| \leq \|\eta\|^k \|b_0\| \cdots \|b_{2k}\| \quad\text{for all $b_0,\dots,b_{2k} \in M_m(\C)$}.
\end{equation}

For $\pi \in \NC_2(2k)$, we put $\kappa_\pi := \eta_\pi(\1_m^{\otimes (2k+1)})$. We claim that
\begin{equation}\label{eq:kappa_pi-bound}
\|\kappa_\pi - \1_m\| \leq \bigg(\sum^{k-1}_{j=0} \|\eta\|^j\bigg) \|\eta(\1_m) - \1_m\| \qquad\text{for all $\pi \in \NC_2(2k)$}.
\end{equation}
For $k=1$, the assertion of \eqref{eq:kappa_pi-bound} is true since $\kappa_{(12)} = \eta(\1_m)$. Now, consider the case $k\geq 2$ and take any $\pi \in \NC_2(2k)$. Then $\pi$ contains a block of the form $(p,p+1)$ for some $1 \leq p \leq 2k-1$. Let $\pi' \in \NC_2(2(k-1))$ be obtained by removing the block $(p,p+1)$ from $\pi$. It follows that
\begin{align*}
\kappa_\pi = \eta_\pi(\1_m^{\otimes (2k+1)}) &= \eta_{\pi'}(\1_m^{\otimes (p-1)} \otimes \eta(\1_m) \otimes \1_m^{\otimes (2k-1-p)})\\
&= \eta_{\pi'}(\1_m^{\otimes (p-1)} \otimes (\eta(\1_m)-\1_m) \otimes \1_m^{\otimes (2k-1-p)}) + \kappa_{\pi'}.
\end{align*}
By using the bound \eqref{eq:eta_pi-bound}, we may deduce from the latter that
$$\|\kappa_\pi-\1_m\| \leq \|\eta\|^{k-1} \|\eta(\1_m)-\1_m\| + \|\kappa_{\pi'}-\1_m\|.$$
Thus, the validity of \eqref{eq:kappa_pi-bound} follows by mathematical induction.

Because $\int^2_{-2} t^{2k}\, d\mu_s(t) = \# \NC_2(2k)$, the formula \eqref{eq:moment_cululant_formula} yields that
$$\left\|\E[S^{2k}] - \int^2_{-2} t^{2k}\, d\mu_s(t)\, \1_m\right\| \leq \sum_{\pi \in \NC_2(2k)} \| \kappa_\pi - \1_m\|.$$
By using the bound \eqref{eq:kappa_pi-bound} as well as the facts that $C_k = \#\NC_2(2k)$ and $|\tr_m(b)| \leq \|b\|$ for all $b\in M_m(\C)$, we arrive at the asserted bound.
\end{proof}

\subsection{The indecomposable case}\label{subsec:proof_thm:1.2_indecomposable}

The goal of this subsection is to verify the assertion of \Cref{thm:1.2} in the case of completely positive linear maps which are indecomposable in the sense of \Cref{def:positive_linear_maps} \ref{it:indecomposable}. The precise result reads as follows.

\begin{proposition}\label{prop:capacity_formula_indecomposable}
Consider the matrix-valued semicircular element $S = a_1 \otimes s_1 + \ldots + a_n \otimes s_n$ with the associated pair $(\eta,\eta^\ast)$ of dual covariance maps $\eta,\eta^\ast: M_m(\C) \to M_m(\C)$. If $\eta$ is indecomposable, then
$$\Delta(S) = \capa(\eta)^{\frac{1}{2m}} e^{-\frac{1}{2}}.$$
\end{proposition}

To begin with, we list in the following lemma some useful properties of the Fuglede-Kadison determinant which we will need not only for the proof of \Cref{prop:capacity_formula_indecomposable} but also for our subsequent considerations.

\begin{lemma}\label{lem:opval_semicricular_operations}
Consider the matrix-valued semicircular element $S = a_1 \otimes s_1 + \ldots + a_n \otimes s_n$ with the associated pair $(\eta,\eta^\ast)$ of dual covariance maps $\eta,\eta^\ast: M_m(\C) \to M_m(\C)$ which are given by
$$\eta(b) = \sum^n_{i=1} a_i b a_i^\ast \qquad\text{and}\qquad \eta^\ast(b) = \sum^n_{i=1} a_i^\ast b a_i.$$
\begin{enumerate}
 \item\label{it:opval_semicricular_operations:operator_scaling} For $c_1,c_2 \in M_m(\C)$, consider the matrix-valued semicircular element
 $$\widetilde{S} := \sum^n_{i=1} (c_1 a_i c_2) \otimes s_i$$
 Then, $(\eta_{c_1,c_2},(\eta^\ast)_{c_2,c_1})$ is the pair of dual covariance maps associated with $\widetilde{S}$ and we have that
$$\Delta(\widetilde{S}) = |\det(c_1)|^{\frac{1}{m}} |\det(c_2)|^{\frac{1}{m}} \Delta(S).$$
\item\label{it:opval_semicricular_operations:hermitization} The $M_{2m}(\C)$-valued semicircular element defined by
$$S^\sym := \sum^n_{i=1} \begin{bmatrix} 0 & a_i\\ a_i^\ast & 0 \end{bmatrix} \otimes s_i$$
has the self-dual covariance map $\eta^\sym: M_{2m}(\C) \to M_{2m}(\C)$ given by
$$\eta^\sym\left(\begin{bmatrix} b_{11} & b_{12}\\ b_{21} & b_{22} \end{bmatrix}\right) = \begin{bmatrix} \eta(b_{22}) & \rho(b_{21})\\ \rho^\ast(b_{12}) & \eta^\ast(b_{11}) \end{bmatrix}$$
with $\rho(b) := \sum^n_{i=1} a_i b a_i$. We call $S^\sym$ the \emph{hermitization} of $S$. Then
$$\Delta(S^\sym) = \Delta(S).$$
\end{enumerate}
\end{lemma}

Before giving the proof of \Cref{lem:opval_semicricular_operations}, the following remark is in order.

\begin{remark}
The covariance map $\eta: M_m(\C) \to M_m(\C), b \mapsto \sum^n_{i=1} a_i b a_i^\ast$ associated with a matrix-valued semicircular element $S=\sum_{i=1}^n a_i\otimes s_i$ in $(M_m(\M),\tr_m \otimes \tau)$ satisfies $\eta(b) = \E[S b S^\ast]$ for all $b \in M_m(\C)$, where $\E: M_m(\M) \to M_m(\C)$ is the conditional expectation onto $M_m(\C) \subseteq M_m(\M)$ which is defined by $\E := \id_{M_m(\C)} \otimes \tau$. Accordingly, the dual $\eta^\ast: M_m(\C) \to M_m(\C), b \mapsto \sum^n_{i=1} a_i^\ast b a_i$ of $\eta$ satisfies $\eta^\ast(b) = \E[S^\ast b S]$ for all $b \in M_m(\C)$.    
\end{remark}

\begin{proof}[Proof of \Cref{lem:opval_semicricular_operations}]
\begin{enumerate}[wide]
 \item It is straightforward to check that the pair of dual covariance maps associated with $\widetilde{S}$ is given by $(\eta_{c_1,c_2},(\eta^\ast)_{c_2,c_1})$. By definition, we have $\widetilde{S} = (c_1 \otimes \1) S (c_2 \otimes \1)$. Hence, by \Cref{prop:FK_determinant_properties} \ref{it:FK_determinant_multiplicative},
 $$\Delta(\widetilde{S}) = \Delta(c_1 \otimes \1) \Delta(S) \Delta(c_2 \otimes \1).$$
 Thus, it remains to note that $\Delta(c\otimes \1) = \Delta(c) = |\det(c)|^{1/m}$ for all $c \in M_m(\C)$ by \Cref{prop:FK_determinant_properties} \ref{it:FK_determinant_tensor} and \Cref{prop:FK_determinant_properties} \ref{it:FK_determinant_matrices}.
 \item Again, it is easy to verify that the self-dual covariance map $\eta^\sym$ is associated with the selfadjoint matrix-valued semicircular element $S^\sym$. Under the natural trace-preserving identification $M_{2m}(\C) \otimes \M \cong M_2(M_m(\C) \otimes \M)$, $S^\sym$ corresponds to $\begin{bmatrix} 0 & S\\ S^\ast & 0 \end{bmatrix}$; in particular, the selfadjoint operators $S^\sym$ and $\begin{bmatrix} 0 & S\\ S^\ast & 0 \end{bmatrix}$ have the same distribution and hence $\Delta(S^\sym) = \Delta\left(\begin{bmatrix} 0 & S\\ S^\ast & 0 \end{bmatrix}\right)$. So
 $$\Delta\left(\begin{bmatrix} 0 & S\\ S^\ast & 0 \end{bmatrix}\right) = \Delta(S)^{\frac{1}{2}}\Delta(S^*)^{\frac{1}{2}}=\Delta(S)$$
 by \Cref{prop:FK_determinant_properties} \ref{it:FK_determinant_matrix} and \Cref{prop:FK_determinant_properties} \ref{it:FK_determinant_adjoint}. \qedhere
 \end{enumerate}
\end{proof}

Note the similarity between \Cref{lem:opval_semicricular_operations} \ref{it:opval_semicricular_operations:operator_scaling} and the formula describing the transformation behavior of capacities under operator scaling as given in Proposition 2.7 in \cite{GGOW19} (see also \cite{Gur2004}); this shows that \Cref{thm:1.2} is consistent with operator scaling.

\begin{proof}[Proof of \Cref{prop:capacity_formula_indecomposable}]
Since $\eta$ is indecomposable, there exists by \Cref{prop:indecomposable_minimizer} a positive definite matrix $c\in M_m(\C)$ with $\det(c)=1$ and $\det(\eta(c))=\capa(\eta)$; moreover, the operator scaling $\widetilde{\eta} := \eta_{\eta(c)^{-1/2},c^{1/2}}$ is doubly stochastic, that is, $\widetilde{\eta}(\1_m)=\1_m$ and $\widetilde{\eta}^\ast(\1_m)=\1_m$. From \Cref{lem:opval_semicricular_operations} \ref{it:opval_semicricular_operations:operator_scaling}, we get that $(\widetilde{\eta},\widetilde{\eta}^\ast)$ is the pair of dual covariance maps of the matrix-valued semicircular element $\widetilde{S} := \sum^n_{i=1} (\eta(c)^{-1/2} a_i c^{1/2}) \otimes s_i$ and that
$$\Delta(S) = \det(\eta(c))^{\frac{1}{2m}} \det(c)^{-\frac{1}{2m}} \Delta(\widetilde{S}) = \capa(\eta)^{\frac{1}{2m}} \Delta(\widetilde{S}).$$
From \Cref{lem:opval_semicricular_operations} \ref{it:opval_semicricular_operations:hermitization}, we learn that the covariance map $\widetilde{\eta}^\sym$ of the hermitization $\widetilde{S}^\sym$ of $\widetilde{S}$ satisfies $\widetilde{\eta}^\sym(\1_{2m})=\1_{2m}$ since $\widetilde{\eta}$ is doubly stochastic. Thus $\widetilde{\eta}^\sym$ is also doubly stochastic and we conclude, by \Cref{prop:doubly_stochastic_semicircular}, that the scalar-valued distribution of $\widetilde{S}^\sym$ is the standard semicircular distribution; hence
$$\Delta(\widetilde{S}) = \Delta(\widetilde{S}^\sym) = e^{-\frac{1}{2}}.$$
In summary, we obtain that $\Delta(S) = \capa(\eta)^{\frac{1}{2m}} e^{-\frac{1}{2}}$, as asserted.
\end{proof}

\subsection{The general case: proof of the lower bound}\label{subsec:proof_thm:1.2_lower_bound}

Consider the matrix-valued semicircular element $S = a_1 \otimes s_1 + \ldots + a_n \otimes s_n$ with the associated pair $(\eta,\eta^\ast)$ of dual covariance maps $\eta,\eta^\ast: M_m(\C) \to M_m(\C)$. Without loss of generality (see \Cref{rem:vanishing_capacity}) we may assume that $\capa(\eta) > 0$. We want to prove that
\begin{equation}\label{eq:thm:1.2_lower_bound}
\Delta(S) \geq \capa(\eta)^{\frac{1}{2m}} e^{-\frac{1}{2}}.
\end{equation}

The strategy is to consider small perturbations of $S$ for which the associated covariance maps are indecomposable in order to reduce the assertion to \Cref{prop:capacity_formula_indecomposable} with the help of \Cref{prop:FK_determinant_lower_semicontinuity}.

To this end, let $S^\flat$ be an $M_m(\C)$-valued semicircular element in $M_m(\C) \otimes \M$ whose (self-dual) covariance map is the \emph{completely depolarizing channel}
$$\eta^\flat:\ M_m(\C) \to M_m(\C), \quad b\mapsto \eta^\flat(b) = \tr_m(b) \1_m$$
and such that $S$ and $S^\flat$ are free with amalgamation over $M_m(\C)$. Then, $S_t := S + \sqrt{t} S^\flat$, for $t\geq 0$, is a matrix-valued semicircular element with associated pair of dual covariance maps $(\eta_t,\eta_t^\ast)$, where $\eta_t: M_m(\C) \to M_m(\C)$ is given by $\eta_t(b) = \eta(b) + t \eta^\flat(b)$. 

Note that $\eta_t$ is strictly positive for all $t>0$ since $\eta_t(b) \geq t \tr_m(b) \1_m$ for all $b \geq 0$. (In fact, this property is the reason for denoting the completely depolarizing channel by $\eta^\flat$ with regard to \Cref{rem:positive_linear_maps} \ref{it:semi-flat_vs_strictly_positive}.) Hence, as explained in \Cref{rem:positive_linear_maps} \ref{it:strict_positivity_implies_indecomposability}, $\eta_t$ is indecomposable for all $t>0$. By \Cref{prop:capacity_formula_indecomposable}, we thus get that
$$\Delta(S_t) = \capa(\eta_t)^{\frac{1}{2m}} e^{-\frac{1}{2}}.$$

Next, we observe that $\eta_t \geq \eta$ which implies $\capa(\eta_t) \geq \capa(\eta)$ by \Cref{prop:capacity_monotonicity}. Consequently, we have by \Cref{prop:FK_determinant_properties} \ref{it:FK_determinant_multiplicative} and \Cref{lem:opval_semicricular_operations} \ref{it:opval_semicricular_operations:hermitization} that
$$\Delta((S_t^\sym)^2) = \Delta(S_t^\sym)^2 = \Delta(S_t)^2 \geq \capa(\eta)^{\frac{1}{m}} e^{-1}.$$

On the other hand, we have that $\|\eta_t^\sym-\eta^\sym\| = t$ which means that $\lim_{t\searrow 0} \|\eta_t^\sym-\eta^\sym\| = 0$; hence, we have $\lim_{t\searrow 0} L(\mu_{S_t^\sym},\mu_{S^\sym}) = 0$ by \cite[Theorem 3.5]{BM2023} (see also \cite[Theorem 1.1]{Mai2022}) with respect to the L\'evy distance $L$. In particular, $(S_t^\sym)^2$ converges in distribution to $(S^\sym)^2$ as $t \searrow 0$.

Moreover, we have that 
$$\|S_t^\sym\| \leq 2\|\eta_t^\sym\|^{1/2}\leq 2(\Vert \eta^\sym\Vert+t)^{1/2}\leq 2(\Vert \eta^\sym\Vert+1)^{1/2},$$
if we restrict to $0 < t\leq 1$; for the first inequality see, for example, the discussion after Eq. (9.5) in \cite{MS17}.

In summary, $((S_t^\sym)^2)_{t\in (0,1]}$ and $(S_0^\sym)^2 = (S^\sym)^2$ satisfy all conditions required in \Cref{prop:FK_determinant_lower_semicontinuity}. This gives
$$\Delta((S^\sym)^2) \geq \limsup_{t\searrow 0} \Delta((S_t^\sym)^2) \geq \capa(\eta)^{\frac{1}{m}} e^{-1}.$$
With the help of \Cref{lem:opval_semicricular_operations} \ref{it:opval_semicricular_operations:hermitization} and \Cref{prop:FK_determinant_properties} \ref{it:FK_determinant_multiplicative}, the latter implies that $\Delta(S) = \Delta(S^\sym) = \Delta((S^\sym)^2)^{\frac{1}{2}} \geq \capa(\eta)^{\frac{1}{2m}} e^{-\frac{1}{2}}$, as asserted in \eqref{eq:thm:1.2_lower_bound}.

\subsection{The general case: proof of the upper bound}\label{subsec:proof_thm:1.2_upper_bound}

Consider again a matrix-valued semicircular element $S = a_1 \otimes s_1 + \ldots + a_n \otimes s_n$ with the associated pair $(\eta,\eta^\ast)$ of dual covariance maps $\eta,\eta^\ast: M_m(\C) \to M_m(\C)$. Without loss of generality (see \Cref{rem:vanishing_capacity}) we may assume that $\capa(\eta) > 0$. We want to prove that
\begin{equation}\label{eq:thm:1.2_upper_bound}
\Delta(S) \leq \capa(\eta)^{\frac{1}{2m}} e^{-\frac{1}{2}}.
\end{equation}

Let $\delta\in (0,\frac{1}{6}]$ and take any positive definite matrix $c_\delta\in M_m(\C)$ such that
$$\capa(\eta) \geq e^{-\delta} \cdot \frac{\det(\eta(c_\delta))}{\det(c_\delta)}.$$
For the associated operator scaling $\widetilde{\eta}_\delta := \eta_{\eta(c_\delta)^{-1/2},c_\delta^{1/2}}$, we have $\ds(\widetilde{\eta}_\delta) \leq 6\delta$ according to \Cref{prop:approximate_minimizer}. From \Cref{lem:opval_semicricular_operations} \ref{it:opval_semicricular_operations:operator_scaling}, we get that $\widetilde{S}_\delta := \sum^n_{i=1} (\eta(c)^{-1/2} a_i c^{1/2}) \otimes s_i$ is a matrix-valued semicircular element the associated pair of covariance maps of which is given by $(\widetilde{\eta}_\delta,\widetilde{\eta}_\delta^\ast)$ and that
$$\Delta(S) = \left(\frac{\det(\eta(c_\delta))}{\det(c_\delta)}\right)^{\frac{1}{2m}} \Delta(\widetilde{S}_\delta) \leq \capa(\eta)^{\frac{1}{2m}} e^{\frac{\delta}{2m}} \Delta(\widetilde{S}_\delta).$$
With \Cref{lem:opval_semicricular_operations} \ref{it:opval_semicricular_operations:hermitization}, we infer for the hermitization $\widetilde{S}_\delta^\sym$ of $\widetilde{S}_\delta$ that
\begin{equation}\label{eq:thm:1.2_upper_bound-1}
\Delta(S) \leq \capa(\eta)^{\frac{1}{2m}} e^{\frac{\delta}{2m}} \Delta(\widetilde{S}_\delta^\sym).
\end{equation}

Moreover, \Cref{lem:opval_semicricular_operations} \ref{it:opval_semicricular_operations:hermitization} tells us that $\widetilde{S}_\delta^\sym$ is a selfadjoint matrix-valued semicircular element with the self-dual covariance map $\widetilde{\eta}_\delta^\sym$ satisfying $\ds(\widetilde{\eta}_\delta^\sym) = 2 \ds(\widetilde{\eta}_\delta)$.
As $\delta \searrow 0$, we have that $\ds(\widetilde{\eta}_\delta) \to 0$ and hence $\ds(\widetilde{\eta}_\delta^\sym) \to 0$, which gives that $\|\widetilde{\eta}_\delta^\sym(\1_{2m}) - \1_{2m}\| \to 0$. From \Cref{prop:approximate_doubly_stochastic_semicircular}, it thus follows that $\widetilde{S}_\delta^\sym$ converges in distribution to a standard semicircular element $s$ as $\delta \searrow 0$. Further, we note that 
$\Vert \widetilde{S}_\delta^\sym \Vert \leq 2\Vert\widetilde\eta_\delta^\sym\Vert^{1/2}
$. Since $\Vert\widetilde\eta_\delta^\sym\Vert = \Vert\widetilde\eta_\delta^\sym(\1_{2m})\Vert$, we have also a uniform bound on the norm of the $\widetilde{S}_\delta^\sym$.
From \Cref{prop:FK_determinant_properties} \ref{it:FK_determinant_doubly_stochastic} and \Cref{prop:FK_determinant_lower_semicontinuity}, it follows that
\begin{equation}\label{eq:thm:1.2_upper_bound-2}
e^{-\frac{1}{2}}= \Delta(s)\geq \limsup_{\delta\searrow0} \Delta(\widetilde{S}_\delta^\sym).
\end{equation}

By putting \eqref{eq:thm:1.2_upper_bound-1} and \eqref{eq:thm:1.2_upper_bound-2} together, we get \eqref{eq:thm:1.2_upper_bound}, as we wished to show.

\section{Proof of the estimate for the capacity}\label{sec:proof_of_thm:1.3}

Let us recall our main statement on the capacity of a completely positive map in the case of integer coefficient matrices. 

\begin{theorem}
Let $\eta$ be a completely positive map with integer coefficients,
$$\eta:\ M_m(\C)\to M_m(\C), \quad b\mapsto \eta(b):=\sum_{i=1}^n a_iba^*_i,\qquad
\text{with $a_i\in M_m(\Z)$.}$$ 
\begin{itemize}
\item
If $\eta$ is rank-decreasing then $\capa(\eta)=0$.
\item
If $\eta$ is rank non-decreasing then 
$\capa(\eta)\geq 1$.
\end{itemize}
\end{theorem}

The first part is actually an equivalence, according to Lemma 4.5 in \cite{Gur2004}, see our \Cref{prop:capacity_positive} A weaker version of the second part was already proved by Garg, Gurvits, Oliveira, and Widgerson; see Theorem 2.18 of \cite{GGOW19}. They had as lower bound  $1/(m^{2m})$ instead of 1. Our proof of the stronger, size-independent, lower bound will follow mainly their ideas, but instead with approximating matrices we will work directly in the limit of Haar unitaries. 

\begin{proof}
We only have to prove the second part. So
assume that $\eta$ is not rank-decreasing, namely, that 
$A=a_1\otimes \fx_1+\ldots+a_n\otimes \fx_n$ is invertible over the free skew field. 
We will now work with an analytic model of this operator; but instead of plugging in free semicirculars for the $\fx_i$ we will choose free Haar unitaries. For this,
let $u_1,\dots,u_n$ be the generators of the free group $\FF_n$, realized as left-multiplication operators in $L(\FF_n)$. With respect to the canonical trace $\tau$ on $L(\FF_n)$, those operators are free Haar unitaries. Now consider the operator
$$T:=\sum_{i=1}^n a_i\otimes u_i\in M_m(L(\FF_n)),$$
where the trace on $M_m(L(\FF_n))=M_m(\C)\otimes L(\FF_n)$ is given by $\tr_m\otimes\tau$.

Since the determinantal conjecture is true for the free group \cite{luck2002l2}, we know that $\Delta(T)\geq 1$. One should note that L\"uck uses a modified version of the Fuglede-Kadison determinant, which ignores the kernel of the operator. Since, by a theorem of Linnell, free Haar unitaries generate the free skew field (see, Corollary 6.4 in \cite{MSY2023}), we know that $T$ is invertible as an unbounded operator and thus has no kernel. Hence, we have $\Delta(T)\geq 1$ for the ordinary Kadison-Fuglede determinant $\Delta$.

Now let us estimate the capacity of $\eta$. Consider a positive $b\in M_m(\C)$ with $\det(b)=1$; hence, $\Delta(b)=1$. We have to show that also $\det(\eta(b))\geq 1$, which is the same as $\Delta(\eta(b))\geq 1$.

Since $\det(b)=1$ and $\eta$ is rank non-decreasing, $b$ and thus also $\eta(b)$ have full rank and hence $\eta(b)$ is invertible. Thus we can consider
$$c_i:=\eta(b)^{-1/2} a_i b^{1/2}\in M_m(\C)$$
and use those to define the new operator
$$\widetilde T:=\sum_{i=1}^n c_i\otimes u_i=\eta(b)^{-1/2} T b^{1/2} \in M_m(L(\FF_n)).$$
Since both $\Delta(T)$ and $\Delta(b^{1/2})=\Delta(b)^{1/2}$ are bigger than 1, we get, by the multiplicativity of $\Delta$, \Cref{prop:FK_determinant_properties} \ref{it:FK_determinant_multiplicative}, the estimate
$$\Delta(\widetilde T)=\Delta\bigl(\eta(b)^{-1/2}\bigr)\cdot \Delta(T) \cdot \Delta( b^{1/2})\geq 
\Delta\bigl(\eta(b)\bigr)^{-1/2}. 
$$
Note that we have
$$\sum_{i=1}^n c_ic_i^*=\sum_{i=1}^n \eta(b)^{-1/2} a_i b^{1/2} b^{1/2} a_i^* \eta(b)^{-1/2}=
\eta(b)^{-1/2}\eta(b) \eta(b)^{-1/2}=\1_m.$$
This allows us to estimate
$\Delta(\widetilde T)=\Delta(\widetilde T \widetilde T^*)^{1/2}$ in the other direction, by using the geometric mean-algebraic mean inequality, \Cref{prop:FK_determinant_properties} \ref{it:FK_AMGM}, and $\tau(u_iu_j^*)=\delta_{ij}$. Indeed, we get in this way
\begin{multline*}
\Delta(\widetilde T \widetilde T^*)\leq \tr_m\otimes\tau(\widetilde T\widetilde T^*)\\
=\sum_{i,j=1}^n \tr_m(c_ic_j^*)\tau(u_iu^*_j) = \sum_{i=1}^n \tr_m(c_ic_i^*) 1=\tr_m(\1_m)=1.
\end{multline*}
Putting our two inequalities together leads then to
$$\Delta\bigl(\eta(b)\bigr)^{- 1/2}\leq \Delta(\widetilde T)=\Delta(\widetilde T \widetilde T^*)^{1/2}\leq
1,$$
and thus
$\Delta(\eta(b))\geq 1$.
Since this is true for any positive $b$ with $\det(b)=1$ we have shown that
$\capa(\eta)\geq 1$.
\end{proof}

\section{Questions and conjectures}

Let us put our estimate for the Fuglede-Kadision determinant of a matrix-valued semicircular operator in the context of the determinantal property of groups.
By rescaling our semicircular variables to $\widetilde s_i:=s_ie^{1/2}$ we can state our \Cref{thm:1.4} in a form which looks exactly like the determinantal property for sofic groups: For all invertible
$\widetilde S=\sum_{i=1}^n a_i\otimes \widetilde s_i$ with $a_i\in M_m(\Z)$
we have that $\Delta(\widetilde S)\geq 1$.

It would of course be nice if one could prove similar statements for situations where we replace the operators $(\widetilde s_1,\dots,\widetilde s_n)$ by other tuples of operators. 

\begin{question}
For which tuples $(x_1,\dots,x_n)$ of operators do we have the following
determinantal property: For all invertible
$X=\sum_{i=1}^n a_i\otimes x_i$ with $m\in\N$ and $a_i\in M_m(\Z)$
is $\Delta(X)\geq 1$?
\end{question}

At the moment we know this for generators of free groups and for free semicircular operators.
However, we do not see a way to deal with this in general. In particular, our proof of \Cref{thm:1.4} relies on the very special property from \Cref{prop:doubly_stochastic_semicircular} of the semicircular situation.

In the group case, the proof of the determinantal property relies on showing the corresponding property for approximating matrices and then extend this to the limit. In the semicircular case there are of course also canonical approximating matrices for $s_1,\dots,s_n$. Those have, however, not integer values as entries, so the basic argument why the determinant is $\geq 1$ cannot be used here. On the other hand there are many investigations about determinants of random matrices (like \cite{zaporozhets2014random,nguyen2014random}), which are compatible with the limiting statement $\Delta\geq 1$. However, as far as we know all these assume some special structure (in particular, quite some independence between the entries) for the matrices. Our setting is much more general and suggests thus the following conjecture for arbitrary Gaussian (or more general Wigner) random matrices.

\begin{conjecture}
Let $m\in\N$ and $a_i\in M_m(\Z)$ be given and denote by $\eta$ the corresponding completely positive map
$$\eta:\ M_m(\C)\to M_m(\C), \quad b\mapsto \eta(b):=\sum_{i=1}^n a_iba^*_i.$$
Assume that $\eta$ is rank non-decreasing. Then consider $n$ independent
GUE random matrices $X^{(N)}_1,\dots,X^{(N)}_n$ and let $A_N$ denote the Gaussian block random matrix
$$A_N=\sum_{i=1}^n a_i\otimes X^{(N)}_i$$
of size $mN\times mN$.
We claim that $A_N$ is invertible for sufficiently large $N$ and that
$$\lim_{N\to\infty}\Delta(A_N)=\capa(\eta)^{\frac 1{2m}}e^{-\frac{1}{2}},$$
both in expectation and almost surely.
\end{conjecture}

\bibliographystyle{amsalpha}
\bibliography{main}

\providecommand{\bysame}{\leavevmode\hbox to3em{\hrulefill}\thinspace}
\providecommand{\MR}{\relax\ifhmode\unskip\space\fi MR }
\providecommand{\MRhref}[2]{%
  \href{http://www.ams.org/mathscinet-getitem?mr=#1}{#2}
}
\providecommand{\href}[2]{#2}
\begin{thebibliography}{GGOW20}

\bibitem[AEK19]{AjEK2019}
O.~H. {Ajanki}, L.~{Erd{\H{o}}s}, and T.~{Kr{\"u}ger}, \emph{Stability of the
  matrix {Dyson} equation and random matrices with correlations}, Probab.
  Theory Relat. Fields \textbf{173} (2019), no.~1-2, 293--373.

\bibitem[BM20]{BM2020}
M.~{Banna} and T.~{Mai}, \emph{H{\"o}lder continuity of cumulative distribution
  functions for noncommutative polynomials under finite free {Fisher}
  information}, J. Funct. Anal. \textbf{279} (2020), no.~8, 44, Id/No 108710.

\bibitem[BM23]{BM2023}
\bysame, \emph{Berry-{Esseen} bounds for the multivariate {{\(\mathcal{B}
  \)}}-free {CLT} and operator-valued matrices}, Trans. Am. Math. Soc.
  \textbf{376} (2023), no.~6, 3761--3818.

\bibitem[CM23]{chatterjee2023noncommutative}
A.~{Chatterjee} and P.~{Mukhopadhyay}, \emph{{The Noncommutative Edmonds'
  Problem Re-visited}}, arXiv preprint arXiv:2305.09984 (2023).

\bibitem[{Den}09]{deninger2009mahler}
C.~{Deninger}, \emph{{Mahler measures and Fuglede--Kadison determinants}},
  arXiv preprint arXiv:0905.0604 (2009).

\bibitem[ES05]{elek2005hyperlinearity}
G.~{Elek} and E.~{Szab{\'o}}, \emph{Hyperlinearity, essentially free actions
  and l 2-invariants. the sofic property}, Mathematische Annalen \textbf{332}
  (2005), 421--441.

\bibitem[FK52]{fuglede1952determinant}
B.~{Fuglede} and R.~V. {Kadison}, \emph{Determinant theory in finite factors},
  Annals of Mathematics (1952), 520--530.

\bibitem[GGOW16]{GGOW16}
A.~{Garg}, L.~{Gurvits}, R.~{Oliveira}, and A.~{Wigderson}, \emph{A
  deterministic polynomial time algorithm for non-commutative rational identity
  testing}, 2016 IEEE 57th Annual Symposium on Foundations of Computer Science
  (FOCS), IEEE, 2016, pp.~109--117.

\bibitem[GGOW20]{GGOW19}
\bysame, \emph{Operator scaling: theory and applications}, Foundations of
  Computational Mathematics \textbf{20} (2020), no.~2, 223--290.

\bibitem[{Gur}04]{Gur2004}
L.~{Gurvits}, \emph{Classical complexity and quantum entanglement}, J. Comput.
  Syst. Sci. \textbf{69} (2004), no.~3, 448--484.

\bibitem[{Hal}74]{Halmos1974}
P.~R. {Halmos}, \emph{Finite-dimensional vector spaces. {Reprint} of the 2nd
  ed}, Undergraduate Texts Math., Springer, Cham, 1974.

\bibitem[{Hay}16]{hayes2016fuglede}
B.~{Hayes}, \emph{Fuglede--kadison determinants and sofic entropy}, Geometric
  and Functional Analysis \textbf{26} (2016), 520--606.

\bibitem[HH21]{hamada2020computing}
M.~{Hamada} and H.~{Hirai}, \emph{Computing the nc-rank via discrete convex
  optimization on {CAT}(0) spaces}, SIAM J. Appl. Algebra Geom. \textbf{5}
  (2021), no.~3, 455--478.

\bibitem[HL00]{haagerup2000brown}
U.~{Haagerup} and F.~{Larsen}, \emph{{Brown's spectral distribution measure for
  R-diagonal elements in finite von Neumann algebras}}, Journal of Functional
  Analysis \textbf{176} (2000), no.~2, 331--367.

\bibitem[HMS23]{hoffmann2023computing}
J.~{Hoffmann}, T.~{Mai}, and R.~{Speicher}, \emph{Computing the noncommutative
  inner rank by means of operator-valued free probability theory}, arXiv
  preprint arXiv:2308.03667 (2023).

\bibitem[HS07]{haagerup2007brown}
U.~{Haagerup} and H.~{Schultz}, \emph{Brown measures of unbounded operators
  affiliated with a finite von neumann algebra}, Mathematica Scandinavica
  (2007), 209--263.

\bibitem[IQS18]{ivanyos2018constructive}
G.~{Ivanyos}, Y.~{Qiao}, and K.~V. {Subrahmanyam}, \emph{Constructive
  non-commutative rank computation is in deterministic polynomial time},
  Comput. Complexity \textbf{27} (2018), no.~4, 561--593.

\bibitem[{Jus}]{sofic}
K.~{Juschenko}, \emph{{Sofic groups}},
  \url{https://web.ma.utexas.edu/users/juschenko/files/soficgroups.pdf}.

\bibitem[L{\"u}c02]{luck2002l2}
W.~L{\"u}ck, \emph{{{\(L^2\)}}-invariants: Theory and applications to geometry
  and {{\(K\)}}-theory}, Ergeb. Math. Grenzgeb., 3. Folge, vol.~44, Berlin:
  Springer, 2002.

\bibitem[{Mai}22]{Mai2022}
T.~{Mai}, \emph{{The Dyson equation for $2$-positive maps and H{\"o}lder bounds
  for the Lévy distance of densities of states}}, {arXiv preprint
  arXiv:2210.04743 [math.FA]} (2022), 27.

\bibitem[{Mir}55]{Mirsky1955}
L.~{Mirsky}, \emph{An introduction to linear algebra}, Oxford: {At} the
  {Clarendon} {Press} {XI}, 433 p. (1955)., 1955.

\bibitem[MS17]{MS17}
J.~A. {Mingo} and R.~{Speicher}, \emph{{Free probability and random
  matrices.}}, vol.~35, Toronto: The Fields Institute for Research in the
  Mathematical Sciences; New York, NY: Springer, 2017.

\bibitem[MSY18]{MSY2018}
T.~{Mai}, R.~{Speicher}, and S.~{Yin}, \emph{{The free field: zero divisors,
  Atiyah property and realizations via unbounded operators}}, {arXiv preprint
  arXiv:1805.04150} (2018).

\bibitem[MSY23]{MSY2023}
\bysame, \emph{The free field: realization via unbounded operators and {Atiyah}
  property}, J. Funct. Anal. \textbf{285} (2023), no.~5, 50, Id/No 110016.

\bibitem[{Nay}18]{nayak2018hadamard}
S.~{Nayak}, \emph{The {Hadamard} determinant inequality -- extensions to
  operators on a {Hilbert} space}, J. Funct. Anal. \textbf{274} (2018), no.~10,
  2978--3002.

\bibitem[NSS02]{nica2002operator}
A.~{Nica}, D.~{Shlyakhtenko}, and R.~{Speicher}, \emph{Operator-valued
  distributions. {I}: {Characterizations} of freeness}, Int. Math. Res. Not.
  \textbf{2002} (2002), no.~29, 1509--1538.

\bibitem[NV14]{nguyen2014random}
H.~H. {Nguyen} and V.~{Vu}, \emph{Random matrices: law of the determinant},
  Ann. Probab. \textbf{42} (2014), no.~1, 146--167.

\bibitem[{Sch}01]{schick2001}
T.~{Schick}, \emph{L$^2$-determinant class and approximation of l$^2$-betti
  numbers}, Transactions of the American Mathematical Society \textbf{353}
  (2001), no.~8, 3247--3265.

\bibitem[SS15]{SS2015}
D.~{Shlyakhtenko} and P.~{Skoufranis}, \emph{{Freely independent random
  variables with non-atomic distributions}}, {Trans. Am. Math. Soc.}
  \textbf{367} (2015), no.~9, 6267--6291.

\bibitem[ZK14]{zaporozhets2014random}
D.~{Zaporozhets} and Z.~{Kabluchko}, \emph{Random determinants, mixed volumes
  of ellipsoids, and zeros of gaussian random fields}, Journal of Mathematical
  Sciences \textbf{199} (2014), 168--173.

\end{thebibliography}

\end{document}